\documentclass[reqno]{amsart}
\usepackage[paper=letterpaper,margin=1in]{geometry}
\usepackage{graphicx,amsmath,amsthm,amsfonts,amssymb,mathrsfs}
\usepackage[colorlinks=true, pdfstartview=FitV, linkcolor=blue, citecolor=blue, urlcolor=blue,pagebackref=false]{hyperref}

\newtheorem{proposition}{Proposition}
\newtheorem{theorem}[proposition]{Theorem}
\newtheorem{lemma}[proposition]{Lemma}
\newtheorem{corollary}[proposition]{Corollary}
\theoremstyle{remark}
\newtheorem{remark}[proposition]{Remark}

\numberwithin{equation}{section}
\numberwithin{proposition}{section}
\allowdisplaybreaks

\renewcommand{\leq}{\leqslant}
\renewcommand{\geq}{\geqslant}

\usepackage{acronym}
\acrodef{SHE}{Stochastic Heat Equation}
\acrodef{WZ}{Wang--Zakai}
\acrodef{LHS}{Left Hand Side}
\acrodef{RHS}{Right Hand Side}

\newcommand{\e}{\varepsilon}
\newcommand{\U}{\mathcal{U}}	%the PDE soln
\newcommand{\calX}{\mathcal{X}}
\newcommand{\Xm}{X^\text{mi}}
\newcommand{\calXm}{\mathcal{X}}
\newcommand{\calXmco}{\mathcal{Z}}
\newcommand{\calXmcoo}{\tilde{\mathcal{Z}}}
\newcommand{\Xmcoo}{\tilde{X}^\text{mi}}

\newcommand{\R}{\mathbb{R}}

\newcommand{\fil}{\mathscr{F}}		%filtration

\renewcommand{\P}{\mathbb{P}}		%probability
\newcommand{\E}{\mathbb{E}}			%expectation
\newcommand{\sgn}{\text{sgn}}		%sign function
\newcommand{\ind}{\mathbf{1}}		%indicator function
\newcommand{\img}{\mathbf{i}}		%imaginary
\newcommand{\la}{\left\langle}
\newcommand{\ra}{\right\rangle}

\newcommand{\ov}{\overline}

\renewcommand{\tilde}{\widetilde}
\renewcommand{\hat}{\widehat}
\usepackage{graphicx}
\newcommand*{\Cdot}{{\raisebox{-0.5ex}{\scalebox{1.8}{$\cdot$}}}} % largedot

\newcommand{\rev}[1]{\textcolor{black}{#1}}

\title[The Wong--Zakai theorem for SHE]{Another look into the Wong--Zakai Theorem\\{} for Stochastic Heat Equation}
\author{Yu Gu, Li-Cheng Tsai}

\address[Yu Gu]{\hspace{29pt}Department of Mathematics, Carnegie Mellon University, Pittsburgh, PA 15213}

\address[Li-Cheng Tsai]{Departments of Mathematics, Columbia University,
2990 Broadway, New York, NY 10027}

\begin{document}
\begin{abstract}
For the heat equation driven by a smooth, Gaussian random potential: 
\begin{align*}
	\partial_t u_\e=\tfrac12\Delta u_\e+u_\e(\xi_\e-c_\e), \   \  t>0, x\in\R,
\end{align*}
where $\xi_\e$ converges to a spacetime white noise, and $c_\e $ is a diverging constant chosen properly, 
we prove that $ u_\e $ converges in $ L^n $ to the solution of the stochastic heat equation for any $n\geq1$. % $\partial_t \U=\tfrac12\Delta \U+\U\xi$.
Our proof is probabilistic, hence provides another perspective of the general result of Hairer and Pardoux \cite{Hairer15a},
for the special case of the stochastic heat equation. We also discuss the transition from homogenization to stochasticity.
\end{abstract}

%\medskip
%\noindent \textsc{MSC 2010:} %82C22, 35B65, 60K35.
%\medskip
%\noindent \textsc{Keywords:} %quantitative homogenization, interacting particle system, tagged particle.
\maketitle
\section{Introduction and main result}
\label{s.intro}

%The study of stochastic PDEs has witnessed significant progress in recent years. 
%Of particular interest is the Wong--Zakai theorems for stochastic PDEs.
%That is, instead of directly approaching a given stochastic PDE driven by white noise $ \xi $,
%one takes a \emph{mollified} noise $ \xi_\e $ and consider the corresponding equation driven by $ \xi_\e $.
%The question is then to study, whether the solution of the mollified equation converges to a universal (i.e., regardless of the mollifier) limit.
%This is the stochastic PDEs analog of the classical result of~\cite{Wong65,Wong65a}.

The study of stochastic PDEs has witnessed significant progress in recent years. Several theories have been developed to make sense of singular equations with multiplication of distributions, see \cite{Hairer13,Hairer14,Gubinelli15,Kupiainen16,Otto16} (and the references therein).
One example is the Wong--Zakai theorem for stochastic PDEs \cite{Hairer15,Hairer15a,Chandra17}, which is an infinite dimensional analogue of~\cite{Wong65,Wong65a,Stroock72}. % (and the references therein).
In this article, we revisit this problem for a special case: the \ac{SHE} in one space dimension:
\begin{align}
	\label{e.she}
	\partial_t \U=\tfrac12\partial_{xx} \U+ \U\xi, \quad t>0, x\in\R,
\end{align}
where $\xi$ is a spacetime white noise, built on an underlying probability space $(\Omega,\fil,\P)$. 

The \ac{SHE}~\eqref{e.she} has played an important role in the study of directed polymers and random growth phenomena.
On the one hand, the solution of~\eqref{e.she} gives the partition function of a directed polymer in a noisy environment.
On the other hand, via the inverse Hopf--Cole transform, the equation~\eqref{e.she} yields the celebrated Kardar--Parisi--Zhang equation \cite{Kardar86},
which describes the height function of a certain type of random growth phenomenon. 
Even though the physical phenomena described by \ac{SHE} goes to higher dimensions, we focus on one dimension here. %, since the \ac{SHE} is not well-defined in higher dimensions.
In fact, $ d= 2 $ and $d>2$ corresponds to the so-called critical and supercritical cases, and sit beyond any existing theory.
There, however, has been works on the cases (in $ d \geq 2 $) where the noise is tuned down to zero suitably with a scaling parameter. 
See~\cite{Feng12,Feng16,Caravenna17,Gu17,Magnen17}.

Throughout this article, 
we fix a bounded continuous initial condition $ u_0(x)\in C_b(\R) $,
and a mollifier $\phi\in C_c^\infty(\R^2;\R_+)$ with a unit total mass $\int \phi \, dt dx=1$.
Using this mollifier, we construct the mollified noise as
\begin{equation}\label{e.defxi}
	\xi_\e(t,x)=\int_{\R^2} \phi_\e(t-s,x-y)\xi(s,y)dyds,
	\quad
	\phi_\e(t,x)=\tfrac{1}{\e^3}\phi(\tfrac{t}{\e^2},\tfrac{x}{\e}).
\end{equation}
Given the smooth function $ \xi_\e $, consider the equation
\begin{equation}\label{e.maineq}
	\partial_t u_\e=\tfrac12\partial_{xx} u_\e+ u_\e(\xi_\e-c_\e), \   \  t>0, x\in\R.
\end{equation}
%Here $ c_\e >0 $ is a positive constant, which will be chosen properly (as in~\eqref{e.c} in the following) for $ u_\e $ to converge.
For the analogous equation where $ \xi_\e $ is white-in-time, regularized in space, and interpreted in the It\^o's sense,
Bertini and Cancrini~\cite{Bertini95} showed that, for $ c_\e=0 $, the solution $ u_\e $ converges to the solution of the \ac{SHE}.

When the noise is regularized in both space and time, a non-zero, \emph{divergent} constant $ c_\e\to\infty $ arises.
Our main result states that, for a suitable and explicit choice of $ c_\e $, which depends explicitly on $\phi$, 
the solution $ u_\e $ of~\eqref{e.maineq} converges pointwisely in $ L^n(\Omega) $ to the solution of \ac{SHE}~\eqref{e.she}, for any $ n \geq 1 $.
It is a classical result that~\eqref{e.she} admits a unique (weak and mild) solution starting from $ \U(0,x) = u_0(x) $.
Also, for fixed $ \e>0 $ and for almost every realization of $ \xi_\e $,
it is standard (by Feynman--Kac formula) to show that the PDE~\eqref{e.maineq} admits a unique classical solution.

%Our main result is
\begin{theorem}\label{t.mainth}
Let $ c_\e=c_*\e^{-1}+\tfrac{1}{2}\sigma_*^2$ with $c_*,\sigma_*$ given by \eqref{e.c1} and \eqref{e.sigma}. %in the following.
 Let $ u_\e $ and $ \U $ denote the respective solutions of~\eqref{e.maineq} and \eqref{e.she}, both with initial condition $ u_0(x) $. 
Then, for any $ (t,x)\in\R_+\times\R $ and $ n \geq 1 $, the random variable $ u_\e(t,x) $ converges in $ L^n(\Omega) $ to $ \U(t,x) $, i.e.,
\begin{align}
	\label{e.main}
	\E\big[ |u_\e(t,x)-\U(t,x)|^n\big] \to 0,
	\quad
	\text{as }\e \to 0.
\end{align}
\end{theorem}

As mentioned earlier, analogs of Theorem~\ref{t.mainth} have already been established in different settings.
Hairer and Pardoux \cite{Hairer15a} established the Wong--Zakai theorem for a general class of semi-linear equations on the torus,
with the \ac{SHE} being a special case.
This result was later extended to non-Gaussian noise by Chandra and Shen \cite{Chandra17},
and the problem on the whole line $ \R $ was studied by Hairer and Labb{\'e} \cite{Hairer15}.
In a related direction, Bailleul, Bernicot, and Frey \cite{Bailleul17}
have studied similar stochastic PDEs via paracontrolled calculus.

All the aforementioned works build on the recently developed theory of regularity structure and paracontrolled calculus \cite{Hairer14,Gubinelli15}.
%These theories are rooted in diagrammatic-calculations, and refine, pathwise analysis of PDEs.
%
In this article, we present a more probabilistic proof of Theorem~\ref{t.mainth}.
The proof is short, entirely contained within the scope of classical stochastic analysis.
This offers a different perspective of the Wong--Zakai theorem for the \ac{SHE}.
For example, the renormalizing constant $ c_\e $ is identified in terms of the first and second moments of certain additive functionals
of Brownian motions. % via the Feynman--Kac formula.
See the discussion in Section~\ref{s.BM}. 
%Specifically, tools involved in this articles include Chaos-expansion, Stroock's formula~\cite[eqn~(7), p3]{Stroock},
%Feynman--Kac formula, martingale central limit theorem, and Clark--Ocone formula.

\rev{%
A related question of interest concerns homogenization of $ \partial_t u_\e=\tfrac12\partial_{xx} u_\e+ \sqrt{\e}u_\e\xi_\e $.
For a general class of mixing random potentials, a homogenization result was established by Pardoux and Piatnitski \cite{Pardoux12}.
We show in Appendix~\ref{s.transition} how our approach can be adopted to establish homogenization.
In fact, we will establish a homogenization result over a range of scales $ \alpha \in [1,2) $,
together with a Gaussian fluctuation result within this range.
}

%{[Our proof also shows that the solution of $\partial_t u_\e=\tfrac12\partial_{xx} u_\e+ \sqrt{\e}u_\e\xi_\e$ converges to $\partial_t \bar{u}=\tfrac12\partial_{xx} \bar{u}+ \bar{u}c_*$, which can be viewed as a homogenization type result---the large, highly-oscillatory random potential $\sqrt{\e}\xi_\e$ is being replaced by the constant $c_*$ in the limit. For a general class of mixing random potentials, the same result was proved by Pardoux and Piatnitski \cite{Pardoux12}. In particular, we will discuss the transition from the homogenization to the convergence to SHE.]}

\subsection*{Outline and conventions}
In Section~\ref{s.BM}, we use the Feynman--Kac formula to analyze moments of $ u_\e(t,x) $.
These formula are expressed in terms of functionals of some auxiliary Brownian motions. 
We then establish various properties of these functionals. 
In Section~\ref{s.proof}, we prove Theorem~\ref{t.mainth} using the results established in Section~\ref{s.BM}.
This is done by first showing that $u_\e(t,x)$ is a Cauchy sequence in $ L^2(\Omega) $,
and then identifying the limit via a Wiener chaos expansion. In Appendix~\ref{s.transition}, we will discuss the homogenization result, prove a central limit theorem, and discuss the transition from the Edwards-Wilkinson to the SHE fluctuations.

Throughout the paper, we denote the Fourier transform of $f$ by 
\begin{align*}
	\hat{f}(\xi)=\int_{\R}f(x)e^{-i\xi x}dx.
\end{align*}
%The heat kernel of $\partial_t-\tfrac12\Delta$ is denoted by $q_t(x)$. 
We use $ C(a_1,\ldots) $ to denote a generic, deterministic, finite constant
that may change from line to line, but depends only on the designated variables $ a_1,\ldots $.
This is \emph{not} to be confused with the renormalization constant $ c_\e $. 
Also, we use $r_\e= (r_\e(t))_{t\geq0} $ to denote a \emph{generic} (random) process, that uniformly converges to zero, i.e.,
\begin{align}
	\label{e.rmd}
	\sup_{t\in\R_+}|r_\e(t)| \leq h_\e \longrightarrow 0,
	\quad
	\text{for some deterministic } h_\e.
\end{align} 

\subsection*{Acknowledgments}
% Thank ppl?
%
YG was partially supported by the NSF through DMS-1613301/1807748 and the Center for Nonlinear Analysis of CMU.
LCT was partially supported by a Junior Fellow award from the Simons Foundation,
and by the NSF through DMS-1712575. We thank the anonymous referee for a very careful reading of the manuscript and many helpful suggestions to improve the presentation.

\section{Feynman--Kac Formula and Brownian Functionals}\label{s.BM}
A main tool in this article is the Feynman--Kac formula, which expresses the solution of the PDE~\eqref{e.maineq} as
\begin{equation}\label{e.fk}
	u_\e(t,x)=\E_B\Big[ u_0(x+B(t))\exp\Big(\int_0^t \xi_\e(t-s,x+B(s))ds-c_\e t \Big) \Big].
\end{equation}
Here $ B(t) $ is a standard Brownian motion starting from the origin, \emph{independent} of the driving noise $ \xi $.
In fact, we will be considering several independent Brownian motions. % $ B $, $ B_1 $, $ B_2 $.
We expand the probability space $ (\Omega,\fil,\P) $
to a larger one $(\Omega\times\Sigma,\fil\times\fil_{B},\P\otimes\P_B)$ 
to include several independent Brownian motions $ B $, $ B_1 $, $ B_2 $,\ldots, $ W, W_1, W_2,\ldots $, independent of $ \xi $.
We will use $ \E_B $ to denote the expectation on $\Sigma$. Also, we will often work with the marginal probability space $ (\Omega,\fil,\P) $ or $ (\Sigma,\fil_B,\P_B) $.

Several functionals of the Brownian motions enter our analysis via \eqref{e.fk}.
More precisely, the $ n $-th moment of $ u_\e(t,x) $ is expressed in terms of functionals of Brownian motions.
We begin with the first moment.
To this end, define the covariance function
\begin{align*}
	R(t,x):= \int_{\R^2} \phi(t-s,x-y) \phi(-s,-y) dsdy,
	\quad
	R_\e(t,x):=\tfrac{1}{\e^3}R(\tfrac{t}{\e^2},\tfrac{x}{\e}) = \E[\xi_\e(t,x)\xi_\e(0,0)].
	\quad
\end{align*}
It is clear that $R$ is an even function. Recall that $ \phi $ is compactly supported.
Without loss of generality, throughout this article we assume that $ \phi $ is supported in $ (-\frac12,\frac12) $ in $ t $,
i.e.,  $ \phi(t,\Cdot)=0$, $|t|\geq \frac12$, and hence $ R(t,\Cdot)=0 $, $ |t|\geq1 $.
With $ \xi_\e(t,x) $ being a Gaussian process, averaging over $ \xi $ in \eqref{e.fk} gives
\begin{align}
\begin{split}
	\E [u_\e(t,x)] &=\E_B\Big[ u_0(x+B(t)) \exp\Big( \frac{1}{2}\E\Big[\Big(\int_0^t \xi_\e(t-s,x+B(s))ds\Big)^2\Big]-c_\e t\Big) \Big]
\\
	\label{e.1mom.}
	&=
	\E_B\Big[ u_0(x+B(t)) \exp\Big( \int_0^t \int_0^s R_\e(s-u,B(s)-B(u)) duds-c_\e t \Big) \Big].
\end{split}
\end{align}
%For our purpose, it is more convenient to work with `microscopic' coordinates.
%That is, we use the scaling property $  (\e^{-1} B(\e^{2}t))_{t\geq 0} \stackrel{\text{law}}{=} (B(t))_{t\geq 0}  $ to rewrite the formula for~$ \E [u_\e(t,x)] $ as
%\begin{align}
%	\label{e.1mom.}
%	\E [u_\e(t,x)] 
%	&=
%	\E_B\Big[ u_0(x+\e B(\e^{-2}t)) \exp\Big( \e\int_0^{\e^{-2}t} \int_0^s R(s-u,B(s)-B(u)) dsdu-C_\e t \Big) \Big].
%\end{align}

Before progressing to the formula for higher moments,
let us use \eqref{e.1mom.} to explain how the renormalizing constant $ c_\e $ comes into play.
To this end, take $ u_0(x)= 1 $ for simplicity.
In this case the solution $ \U $ of the limiting \ac{SHE}~\eqref{e.she} satisfies $ \E[\,\U(t,x)]= 1 $.
For the convergence in \eqref{e.main} to hold, we must choose $ c_\e $ so that
\begin{align*}
	\E_B\Big[ \exp\Big( \int_0^{t} \int_0^s R_\e(s-u,B(s)-B(u)) duds-c_\e t \Big) \Big]\longrightarrow  1,
	\quad
	\text{as } \e \to 0.
\end{align*}
To this end, consider the centered double-integral process
\begin{align}
	\label{e.xeps}
	X_{\e}(t)&:=\int_0^{t}\int_0^s R_\e(s-u,B(s)-B(u))duds- \int_0^{t}\int_0^s \E_B[R_\e(s-u,B(s)-B(u))]duds.
\end{align}
It is more convenient to express $ X_\e $ in `microscopic' coordinates.
That is, we use the scaling property $  (\e^{-1} B(\e^{2}t))_{t\geq 0} \stackrel{\text{law}}{=} (B(t))_{t\geq 0}  $ to write
\begin{align}
	\notag
	X_\e(t) \stackrel{\text{law}}{=} \Xm_\e(t)
	&:=\e \int_0^{\e^{-2}t}\int_0^s R(s-u,B(s)-B(u))duds-\e \int_0^{\e^{-2}t}\int_0^s \E_B[R(s-u,B(s)-B(u))]duds
\\
	\label{e.xm}
	&=
	\e\int_0^{\e^{-2}t} \left(\int_0^s \big(R(u,B(s)-B(s-u))-\E_B[R(u,B(s)-B(s-u))]\big)du\right) ds.
\end{align}
Since $R(u,x)=0$ whenever $|u|\geq 1$, the $ u $-integral in~\eqref{e.xm} goes over $ u\in[0,1] $ for all $ s \geq 1 $.
Dropping those values of $ s<1 $ in the integral gives
\begin{align}
	\label{e.xm.int}
	\Xm_{\e}(t)=\e\int_1^{\e^{-2}t}  \calXm(s) ds + r_\e(t),
\end{align}
where, recall that, $ r_\e(t) $ denotes a generic process satisfying~\eqref{e.rmd},
\begin{align}
	\label{e.calxm}
	\calXm(s) &:= \int_0^\infty R(u,B(s)-B(s-u)) du - c_*,
\\	
	\label{e.c1}
	c_* &:= \E_B \Big[ \int_0^{\infty} R(u,B(u)) du  \Big]
	=
	\E_B \Big[ \int_0^{1} R(u,B(s)-B(s-u)) du  \Big],
\end{align}
and we take $ B $ to be a two-sided Brownian motion in~\eqref{e.calxm}--\eqref{e.c1}
so that the resulting expression is defined for all $ s \geq 0 $ (\emph{including} $ s\in[0,1] $).
It is straightforward to verify that
\begin{align}
	\label{e.calxm.station}
	&\{\calXm(s)\}_{s \geq 0} \text{ is stationary in $s$, bounded, } \E_B[\calXm(s)]=0,
\\
	\label{e.calxm.finitRg}
	&
	\text{with }
	(\calXm(s))_{s \geq s_0} , (\calXm(s))_{s \leq s'_0} \text{ being independent whenever } s_0-s_0' \geq 1 .
\end{align}
Thus it is natural to expect $ X_{\e}(t) $ to converges to $ \sigma_* W(t) $, with $W$ being a standard Brownian motion,
and 
\begin{align}
	\label{e.sigma}
	\sigma^2_* := 2\,\E_B\Big[\int_0^\infty \calXm(s)\calXm(0) ds\Big].
\end{align}
In light of these discussions, we find that $ c_\e := c_*\e^{-1} + \tfrac{1}{2}\sigma_*^2 $ (as in Theorem~\ref{t.mainth})
is the reasonable choice in order for $ u_\e $ to converge to the solution $ \U $ to the~\ac{SHE}.
\rev{\begin{remark}
The constants $ c_* $ and $ \sigma_*^2 $, defined in \eqref{e.c1} and \eqref{e.sigma}, 
can also be expressed in terms of integrals involving the covariance function $ R $ and the heat kernel. 
For example,
\begin{align*}
	c_* =  \int_0^{\infty} \E_B \Big[ R(u,B(u)) \Big] du = \int_0^{\infty} \int_\R R(u,x) \frac{1}{\sqrt{2\pi u}} e^{-\frac{x^2}{2u}} dx du.
\end{align*}
Also, the process $ X_\e(t) $ can be expressed as
\begin{align*}
	X_\e(t) = \int_0^{t}\int_0^s R_\e(s-u,B(s)-B(u))duds-\frac{c_*t}{\e}+r_\e(t).
\end{align*}
\end{remark}}

With $ X_\e(t) $ and $ \sigma_* $ defined in the preceding, we rewrite the formula~\eqref{e.1mom.} in a more compact form as
\begin{align}
	\label{e.1mom}
	\E [u_\e(t,x)] 
	&=
	\E_B\Big[ u_0(x+B(t)) \exp\Big( X_\e(t) - \tfrac{1}{2}\sigma_*^2 t + r_\e(t) \Big) \Big].
\end{align}
Similar calculations give formulas of higher moments:
\begin{align}
	\label{e.nmom}
	\E [u_\e(t,x)^n] 
	&=
	\E_B\Big[ \prod_{j=1}^n u_0(x+ B_j(t)) 
	\exp\Big( \sum_{j=1}^n\Big(X_{j,\e}(t) - \frac{1}{2}\sigma_*^2 t + r_{\e}(t) \Big)
		+\sum_{1\leq i<j\leq n} Y_{i,j,\e}(t)
	\Big) \Big].
\end{align}
Here $ B_1,\ldots,B_n $ are independent Brownian motions;
the process $ X_{j,\e}(t) $ is obtained by replacing $ B $ with $ B_j $ in~\eqref{e.xeps};
and $ Y_{i,j,\e}(t) $ is given by
\begin{align}
	\label{e.ye}
	Y_{i,j,\e}(t)&:=\int_0^t \int_0^tR_{\e}(s-u,B_i(s)-B_j(u))dsdu.
\end{align}
Indeed, $ (X_{j,\e}(t))_{t \geq 0} \stackrel{\text{law}}{=} (X_{\e}(t))_{t \geq 0} $.
Likewise, writing $ Y_{\e}(t) := Y_{1,2,\e}(t) $, we have
$ (Y_{i,j,\e}(t))_{t \geq 0} \stackrel{\text{law}}{=} (Y_{\e}(t))_{t \geq 0} $, for all $ i<j $.

We will also need to consider $ \E[u_{\e_1}(t,x)u_{\e_2}(t,x)] $, i.e., the second moment calculated at different values of $ \e $.
A similar calculation gives the formula
\begin{align}
	\label{e.2mom}
	\E[u_{\e_1}(t,x)u_{\e_2}(t,x)]
	=
	\E_B\Big[ \prod_{j=1}^2 u_0(x+ B_j(t)) 
	\exp\Big( Y_{\e_1,\e_2}(t)+\sum_{j=1}^2\Big(X_{j,\e_j}(t) - \frac{1}{2}\sigma_*^2 t + r_{\e_j}(t) 
	\Big) \Big],
\end{align}
where
\begin{align}
	\label{e.yeps}
	Y_{\e_1,\e_2}(t)&:=\int_0^t\int_0^tR_{\e_1,\e_2}(s-u,B_1(s)-B_2(u))dsdu,
\\
	\notag
 	R_{\e_1,\e_2}(t,x) &:= \int_{\R^2}\phi_{\e_1}(t-s,x-y) \phi_{\e_2}(-s,-y) dsdy = \E[\xi_{\e_1}(t,x)\xi_{\e_2}(0,0)] .
\end{align}

\subsection{Exponential moments}
We first establish bounds on exponential moments of $ X_\e(t) $ and $ Y_{\e_1,\e_2}(t) $.
\begin{proposition}\label{p.expin}
For any $\lambda,t>0$, we have 
\[
\sup_{\e\in(0,1)}\E_B[ e^{\lambda X_\e(t)}]+\sup_{\e_1,\e_2\in(0,1)} \E_B[e^{\lambda Y_{\e_1,\e_2}(t)}]<\infty.
\]
\end{proposition}

\begin{proof}
For $X_\e(t)$, we appeal to the microscopic coordinates,
using~\eqref{e.xm}--\eqref{e.xm.int} to write
\begin{align*}
	X_\e(t)\stackrel{\text{law}}{=} \Xm_\e(t)
	=\e\int_1^{\e^{-2}t}  \calXm(s) ds + r_\e(t)
	=\e\int_1^{[\e^{-2}t]}  \calXm(s) ds + r_\e(t).
\end{align*}
In view of the finite range property~\eqref{e.calxm.finitRg} of $ \calXm $, we decompose 
\begin{align*}
	\Xm_\e(t)
	=
	\e \sum_{k\in I_\text{even}} \tilde{\calXm}_k +\e \sum_{k \in I_\text{odd}}\tilde{\calXm}_k,
\end{align*}
where $  \tilde{\calX}_k:=\int_{k}^{k+1}\calXm(s)ds $,
and $ I_\text{even} := \{ 1\leq k \leq [\e^{2}t]-1, \text{ even} \} $,
and $ I_\text{odd} := \{ 1\leq k \leq [\e^{2}t]-1, \text{ odd} \} $.
This gives
\begin{align*}
	\E_B[e^{\lambda X_\e(t)}]
	&= 	
	\E_B\Big[e^{\lambda r_\e(t)} \exp\Big(\lambda \e \sum_{k\in I_\text{even}}\tilde{\mathcal{X}}_k\Big)\exp\Big(\lambda \e \sum_{k\in I_\text{odd}}\tilde{\mathcal{X}}_k\Big)\Big]
\\
	&\leq C(\lambda,t) \sqrt{\E_B\Big[ \exp\Big(2\lambda \e \sum_{k\in I_\text{even}}\tilde{\mathcal{X}}_k\Big)\Big] \E_B\Big[\exp\Big(2\lambda \e \sum_{k\in I_\text{odd}}\tilde{\mathcal{X}}_k\Big)\Big]}
	= C(\lambda,t) \sqrt{
			\prod_{k \leq [t/\e^{2}]-1 }\E_B[ e^{2\lambda \e \tilde{\mathcal{X}}_k}]
		}.
\end{align*}
By~\eqref{e.calxm.station}, we know that $ \tilde{\calXm}_k $ is stationary with zero mean and is also uniformly bounded since $R\in C_c^\infty$, so
\begin{align*}
	\prod_{k \leq [t/\e^{2}]-1 }\E_B[ e^{2\lambda\e \tilde{\mathcal{X}}_k}]
	=
	\Big( \E_B[ e^{2\lambda\e \tilde{\mathcal{X}}_1}] \Big)^{[t/\e^{2}]-1}
	\leq
	C(\lambda,t).
\end{align*}
From this we conclude the desired exponential moment bound on $ X_\e(t) $:
\begin{align*}
	\sup_{\e\in(0,1)}\E_B[e^{\lambda X_\e(t)}] \leq C(\lambda, t)<\infty.
\end{align*}

We now turn to $Y_{\e_1,\e_2}(t)$. 
Denoting by $\hat{R}_{\e_1,\e_2}$ the Fourier transform of $R_{\e_1,\e_2}$in the $x$-variable,
we express $ Y_{\e_1,\e_2}(t) $ via Fourier transform as
\begin{align*}
	Y_{\e_1,\e_2}(t)=\int_{[0,t]^2} \left(\int_\R (2\pi)^{-1}\hat{R}_{\e_1,\e_2}(s-u,\xi)e^{\img\xi (B_1(s)-B_2(u))}d\xi\right) dsdu.
\end{align*}
Note that $Y_{\e_1,\e_2}\geq0$, so from the above expression we calculate the $ n $-th moment of $ Y_{\e_1,\e_2}(t) $ as 
\begin{align}
\label{e.Ynthmom}
\begin{split}
	\E_B[Y_{\e_1,\e_2}(t)^n]
	=
	\int_{[0,t]^{2n}\times\R^n}
	\prod_{j=1}^n &\big( (2\pi)^{-1}\hat{R}_{\e_1,\e_2}(s_j-u_j,\xi_j) \big)
\\
	&\times\E_B\Big[ \prod_{j=1}^ne^{\img\xi_jB_1(s_j)}\Big] \E_B\Big[ \prod_{j=1}^n e^{-\img\xi_j B_2(u_j)}\Big]
	ds d u d\xi.
\end{split}
\end{align}
Let us first focus on the integral over $ u\in[0,t]^n $. We write
%\begin{align*}
%	\E_B[Y_{\e_1,\e_2}(t)^n]
%	=&
%	\int_{[0,t]^{n}\times\R^n}
%	\Big(
%		\int_{[0,t]^n}
%		\prod_{j=1}^n \big( (2\pi)^{-1}\hat{R}_{\e_1,\e_2}(u-s,\xi_j) \big)
%		\E_B\Big[ \prod_{j=1}^n e^{-\img\xi_j B_2(u)}\Big]
%		d u 
%	\Big)
%	\E_B\Big[ \prod_{j=1}^ne^{\img\xi_jB_1(s)}\Big] 
%	ds d\xi.
%\end{align*}
\begin{align*}
	\int_{[0,t]^n} 
		\prod_{j=1}^n \hat{R}_{\e_1,\e_2}(s_j-u_j,\xi_j)
		\E_B\Big[ \prod_{j=1}^n e^{-\img\xi_j B_2(u_j)}\Big] du
	=
	\int_{[0,t]^n\times\R^n} 
		\prod_{j=1}^n R_{\e_1,\e_2}(s_j-u_j,x_j) e^{-\img \xi_jx_j}
		\E_B\Big[ \prod_{j=1}^n e^{-\img\xi_j B_2(u_j)}\Big]
	dudx.
\end{align*}
The exponents are purely imaginary.
We hence bound those exponentials by $ 1 $ in absolute value, 
and use $ 0\leq \int_{[0,t]\times\R} R_{\e_1,\e_2}(s-u,x) dudx \leq 1 $ to get
\begin{align*}
	\Big|
	\int_{[0,t]^n} 
		\prod_{j=1}^n \hat{R}_{\e_1,\e_2}(u_j-s_j,\xi_j)
		\E_B\Big[ \prod_{j=1}^n e^{-\img\xi_j B_2(u_j)}\Big] du
	\Big|
	\leq
	1.
\end{align*}
Inserting this into~\eqref{e.Ynthmom} gives
\begin{align}
	\label{e.Ynthmom.}
	\E_B[Y_{\e_1,\e_2}(t)^n]
	\leq
	(2\pi)^{-n}
	\int_{[0,t]^{n}\times\R^n} \hspace{-3pt}  \E_B\Big[ \prod_{j=1}^ne^{\img\xi_jB_1(s_j)}\Big] ds d\xi.
%=
%	(2\pi)^{-n}
%	\int_{[0,t]^{n}\times\R^n} \hspace{-8pt} \E_B\Big[ \prod_{j=1}^ne^{\img\xi_jB_1(s)}\Big] ds d\xi.
\end{align}

The last integral in~\eqref{e.Ynthmom.} is in fact the $ n $-th moment of Brownian localtime at the origin.
More precisely, let $ L(t,x;B) $ denote the localtime process of a Brownian motion $ B $, it is a standard result that
\begin{align}
	\label{e.localnmom}
	(2\pi)^{-n}
	\int_{[0,t]^{n}\times\R^n} \hspace{-8pt} \E_B\Big[ \prod_{j=1}^ne^{\img\xi_jB_1(s_j)}\Big] ds d\xi
	= \E_{B} \Big[ L(t,0;B_1)^n \Big].
\end{align}
Informally speaking, this formula is obtained by interpreting $ L(t,0;B_1) $ as $ \int_0^t \delta(B_1(s)) ds $,
where $ \delta(\Cdot) $ denotes the Dirac function, and taking Fourier transform, similarly to the preceding.
The prescribed informal procedure is rigorously implemented by taking a sequence approximating the Dirac function.
We omit the details here as the argument is standard.

Now, combine~\eqref{e.Ynthmom.}--\eqref{e.localnmom}, and sum over $ n \geq 0 $. We arrive at
$ 	
	\E_B[e^{\lambda Y_{\e_1,\e_2}(t)}]
	\leq
	\E_B[e^{\lambda L(t,0;B_1)}].
$
As the Brownian localtime has finite exponential moments, i.e., $ \E[e^{\lambda L(t,0;B_1)}]<\infty $ for any $\lambda,t>0$,
we obtain the desired exponential moment bound on $ Y_{\e_1,\e_2}(t
) $:
\begin{align*}
	\sup_{\e_1,\e_2\in(0,1)}\E_B[e^{\lambda Y_{\e_1,\e_2}(t)}] \leq C(\lambda,t)<\infty.
\end{align*}
This completes the proof.
\end{proof}

\subsection{Weak convergence}
In this section, we derive the distributional limit of $ X_{j,\e} $ and $ Y_{\e_1,\e_2} $.
First, since the covariance function $ R_{\e_1,\e_2}(t,x) $ converges to Dirac function $ \delta(t)\delta(x) $,
we expect the process $ Y_{\e_1,\e_2}(t) $ (defined in~\eqref{e.yeps}) to converge to the mutual intersection localtime of $ B_1 $ and $ B_2 $.
More precisely, recalling that $ L(t,x;B) $ denote the localtime process of a Brownian motion $ B $,
we define the mutual intersection localtime of $ B_1 $ and $ B_2 $ as
\begin{equation}
	\label{e.localtime}
	\ell(t):=L(t,0;B_1-B_2).
\end{equation}
\begin{proposition}\label{p.ycnvg}
For any fixed $t>0$, $ Y_{\e_1,\e_2}(t) \to \ell(t) $ in $ L^2(\Sigma) $, as $\e_1,\e_2\to0 $.
%That is, $ \E_B(Y_{\e_1,\e_2}(t)-\ell(t))^2\to 0 $.
\end{proposition}
\begin{proof}
Instead of directly proving the convergence of $ Y_{\e_1,\e_2}(t)$,
let us first consider a modified process $ \tilde{Y}_{\e_1,\e_2}(t) $ where $ B_1 $ and $ B_2 $ are evaluated at the same time:
\begin{align*}
	\tilde{Y}_{\e_1,\e_2}(t):=\int_0^t\int_0^t R_{\e_1,\e_2}(s-u,B_1(s)-B_2(s))dsdu,
\end{align*}
and show that $ \tilde{Y}_{\e_1,\e_2}(t) $ converges to $ \ell(t) $ in $L^2(\Sigma)$.
To this end, set
\begin{align*}
	F_{\e_1,\e_2}(s):=\int_0^t R_{\e_1,\e_2}(s-u,B_1(s)-B_2(s))du,  \quad  s\in[0,t].
\end{align*}
In the preceding double-integral expression of~$ \tilde{Y}_{\e_1,\e_2}(t) $,
divide the range of integration over $ s $ into subintervals depending on its distance from $ 0 $ and to $ t $.
We rewrite the expression as
\begin{align*}
	\tilde{Y}_{\e_1,\e_2}(t)
	=
	\int_0^t F_{\e_1,\e_2}(s)\Big(\ind_{(\e_1^2+\e_2^2,t-\e_1^2-\e_2^2)}(s)+\ind_{[0,\e_1^2+\e_2^2]}(s)+\ind_{[t-\e_1^2-\e_2^2,t]}(s)\Big)ds.
\end{align*}
Recall that $ \phi(t,\Cdot)=0$, $|t|\geq \frac12$. 
This gives $R_{\e_1,\e_2}(s-u,\Cdot)=0$ for all $|s-u|\geq \e_1^2+\e_2^2$. 
Consequently, for $ s\in (\e_1^2+\e_2^2,t-\e_1^2-\e_2^2) $ we have
\begin{align*}
	F_{\e_1,\e_2}(s)=\int_{\R}R_{\e_1,\e_2}(s-u,B_1(s)-B_2(s))du.
%	=
%	\Phi_{\e_1,\e_2}(B_1(s)-B_2(s)). 
	%=\int_{\R}\Phi_{\e_1,\e_2}(x)L(t,x;B_1-B_2)dx
	%=\int_{\R^2} \Phi(x)\Phi(-y)L(t,\e_1x+\e_2y;B_1-B_2)dxdy
\end{align*}
Further setting $ \Phi(x) := \int_\R \phi(t,x) dt $, $ \Phi_{\e}(x) := \e^{-1}\Phi(\e^{-1}x) $,
and $\Phi_{\e_1,\e_2}(x):=\int \Phi_{\e_1}(x-y)\Phi_{\e_2}(-y)dy$,
we rewrite the last expression as
$ 	
	F_{\e_1,\e_2}(s)= \Phi_{\e_1,\e_2}(B_1(s)-B_2(s)). 
$
On the other hand, we also have $ |F_{\e_1,\e_2}(s)|\leq \Phi_{\e_1,\e_2}(B_1(s)-B_2(s)) $, for all $ s\in[0,t] $. 
This takes into account those values of $ s\not\in (\e_1^2+\e_2^2,t-\e_1^2-\e_2^2) $, thereby giving
\begin{align}
	\label{e.tildey}
	\tilde{Y}_{\e_1,\e_2}(t)=\int_{\e_1^2+\e_2^2}^{t-\e_1^2-\e_2^2} \Phi_{\e_1,\e_2}(B_1(s)-B_2(s))ds +r_{\e_1,\e_2}(t),
\end{align}
where $ r_{\e_1,\e_2}(t) $ is a  remainder term satisfying 
\begin{align*}
%	\label{e.r12}
	|r_{\e_1,\e_2}(t)| \leq \int_0^t \Phi_{\e_1,\e_2}(B_1(s)-B_2(s))\Big(\ind_{[0,\e_1^2+\e_2^2]}(s)+\ind_{[t-\e_1^2-\e_2^2,t]}(s)\Big)ds.
\end{align*}

Now, for any interval $[a,b]\subset [0,\infty)$, by the definition of the localtime, 
\begin{equation}\label{e.phiab}
\begin{aligned}
&\int_a^b \Phi_{\e_1,\e_2}(B_1(s)-B_2(s)) ds=\int_{\R}\Phi_{\e_1,\e_2}(x)\Big(L(b,x;B_1-B_2)-L(a,x;B_1-B_2)\Big)dx\\
=&\int_{\R^2}\Phi(x)\Phi(-y)\Big(L(b,\e_1x+\e_2y;B_1-B_2)-L(a,\e_1x+\e_2y;B_1-B_2)\Big)dxdy.
\end{aligned}
\end{equation}
For almost every realization of $B_1-B_2$, 
the function $ x\mapsto L(t,x;B_1-B_2) $ is continuous and compactly supported,
and the function $ t\mapsto L(t,x;B_1-B_2)$ is increasing and continuous. 
Thus, from \eqref{e.tildey}--\eqref{e.phiab} and the fact that $ \int \Phi dx=1 $, 
we conclude that $\tilde{Y}_{\e_1,\e_2}(t)\to \ell(t)=L(t,0;B_1-B_2)$ almost surely as $\e_1,\e_2\to0$.
Further, the same calculations (via Fourier transform) as in the proof Proposition~\ref{p.expin} 
yields that 
\begin{align*}
	 \sup_{\e_1,\e_2}\E_B\big[e^{\lambda\tilde{Y}_{\e_1,\e_2}(t)} \big]
	 \leq
	 \E_B\big[ e^{\lambda L(t,0;B_1)} \big] < \infty.
\end{align*}
This property leverages the preceding almost sure convergence into a convergence in $ L^2(\Sigma) $:
\begin{align}
	\label{e.tilY.cnvg}
	\E_B[(\tilde{Y}_{\e_1,\e_2}(t)-\ell(t))^2]\longrightarrow 0 \quad\text{as } \e_1,\e_2\to 0 .
\end{align}

Given~\eqref{e.tilY.cnvg}, it remains to show that $ Y_{\e_1,\e_2}(t)-\tilde{Y}_{\e_1,\e_2}(t) \to 0 $ in $ L^2(\Sigma) $. We will actually prove the following result: for any choice of $Z_1,Z_2\in \{Y_{\e_1,\e_2}(t),\tilde{Y}_{\e_1,\e_2}(t)\}$, as $\e_1,\e_2\to0$, 
\begin{equation}\label{e.conz1z2}
	\E_B[Z_1Z_2]
	\longrightarrow 
	(2\pi)^{-2}\int_{[0,t]^2}\int_{\R^2} \E_B\Big[ e^{\img \xi(B_1(s)-B_2(s))} e^{\img \xi'(B_1(s')-B_2(s'))} \Big] d\xi d\xi' ds ds'.
\end{equation}
Once this is done, expanding $ \E[( Y_{\e_1,\e_2}(t)-\tilde{Y}_{\e_1,\e_2}(t))^2] $ into four terms, and passing to the limit complete the proof.
% 
%Note that the RHS of \eqref{e.conz1z2} can be computed explicitly, and we only leave it in that form to facilitate the argument. 

The proof for all cases of $ Z_1,Z_2\in \{Y_{\e_1,\e_2}(t),\tilde{Y}_{\e_1,\e_2}(t)\} $ is the same,
and we take $Z_1=Z_2=Y_{\e_1,\e_2}(t)$ as an example.
As in the proof of Proposition~\ref{p.expin}, we express $Y_{\e_1,\e_2}(t)$ via Fourier transform as
\begin{align*}
	Y_{\e_1,\e_2}(t)
	=&\int_{[0,t]^2} \left(\int_\R (2\pi)^{-1}\hat{R}_{\e_1,\e_2}(s-u,\xi)e^{\img\xi (B_1(s)-B_2(u))}d\xi\right) dsdu
\\
	=&\int_{[0,t]^2} 
	\left(
		\int_{\R^2} (2\pi \e_1^2\e_2^2)^{-1}\hat{\phi}(\tfrac{s-u-w}{\e_1^2},\e_1\xi)\hat{\phi}(\tfrac{-w}{\e_2^2},-\e_2\xi)e^{\img\xi (B_1(s)-B_2(u))}d\xi dw
	\right) dsdu,
\end{align*}
where $\hat{\phi}$ denotes the Fourier transform of $\phi$ in the $x$-variable. 
Squaring the last expression and taking expectation gives
\begin{align*}
%\E_B[Y_{\e_1,\e_2}(t)^2]=\int_{[0,t]^4}\int_{\R^2} &(2\pi \e_1^2\e_2^2)^{-2} \hat{\phi}(\frac{s-u-w}{\e_1^2},\e_1\xi)\hat{\phi}(\frac{-w}{\e_2^2},-\e_2\xi)\hat{\phi}(\frac{s'-u'-w'}{\e_1^2},\e_1\xi')\hat{\phi}(\frac{-w'}{\e_2^2},-\e_2\xi')\\
%&\times \E_B\Big[e^{\img\xi (B_1(s)-B_2(u))}e^{\img\xi' (B_1(s')-B_2(u'))}\Big]d \xi d\xi' dudu'dwdw'dsds'.
	\E_B[Y_{\e_1,\e_2}(t)^2]=\int_{[0,t]^2} G_{\e_1,\e_2}(s,s')dsds',
\end{align*}
where
\begin{align*}
	G_{\e_1,\e_2}(s,s'):=\int_{[0,t]^2}\int_{\R^4} &(2\pi \e_1^2\e_2^2)^{-2} \hat{\phi}(\tfrac{s-u-w}{\e_1^2},\e_1\xi)\hat{\phi}(\tfrac{-w}{\e_2^2},-\e_2\xi)\hat{\phi}(\tfrac{s'-u'-w'}{\e_1^2},\e_1\xi')\hat{\phi}(\tfrac{-w'}{\e_2^2},-\e_2\xi')
\\
	&\times\E_B\Big[e^{\img\xi (B_1(s)-B_2(u))}e^{\img\xi' (B_1(s')-B_2(u'))}\Big]d \xi d\xi' dwdw'dudu'.
\end{align*}
\emph{Fix} $ (s,s')\in(0,t)^2 $.
Recall that $\hat{\phi}(t,\Cdot)=0$ whenever $|t|\geq\tfrac12$,
and note that, with $ (s,s')\in(0,t)^2 $ being fixed,
the conditions
$
	(t-s),(t-s'),s,s' \geq \frac12(\e_1^2+\e_2^2)
$
holds for all small enough $ \e_1,\e_2 $.
Consequently, in the last expression of $ G_{\e_1,\e_2}(s,s') $,
for all $ \e_1,\e_2 $ small enough, the integration domain of $u,u'$ can be (and is) replaced $[0,t]^2\mapsto \R^2$.
A change of variables in this case yields
\begin{align*}
	G_{\e_1,\e_2}(s,s')=\int_{\R^6} &(2\pi)^{-2} \hat{\phi}(u,\e_1\xi)\hat{\phi}(-w,-\e_2\xi)\hat{\phi}(u',\e_1\xi')\hat{\phi}(-w',-\e_2\xi')
\\
	&\times \E_B\Big[ e^{\img \xi (B_1(s)-B_2(s-\e_1^2 u-\e_2^2 w))} e^{\img \xi'(B_1(s')-B_2(s'-\e_1^2u'-\e_2^2 w'))}\Big] d\xi d\xi' dwdw'dudu'.
\end{align*}
Observe that the expectation in the above expression has Gaussian tails in $\xi,\xi'$, so by the dominated convergence theorem and the fact that $\int \hat{\phi}(u,0)du =\int \phi dudx=1$, we obtain
\begin{equation}\label{e.conG}
	G_{\e_1,\e_2}(s,s')\longrightarrow (2\pi)^{-2}\int_{\R^2}\E_B\Big[ e^{\img \xi (B_1(s)-B_2(s))} e^{\img \xi'(B_1(s')-B_2(s'))}\Big] d\xi d\xi',
	\quad
	\text{pointwisely in }(0,t)^2.
\end{equation}

To achieve~\eqref{e.conz1z2}, we need to upgrade the pointwise convergence of~\eqref{e.conG} to convergence in $ L^1([0,t]^2) $.
To this end, with $|\hat{\phi}(\Cdot,\xi)|\leq \hat{\phi}(\Cdot,0)$, we bound
\begin{align*}
|G_{\e_1,\e_2}(s,s')|\leq \int_{\R^6} &(2\pi \e_1^2\e_2^2)^{-2} \hat{\phi}(\tfrac{s-u-w}{\e_1^2},0)\hat{\phi}(\tfrac{-w}{\e_2^2},0)\hat{\phi}(\tfrac{s'-u'-w'}{\e_1^2},0)\hat{\phi}(\tfrac{-w'}{\e_2^2},0)
\\
&\times \E_B\Big[e^{\img\xi B_1(s)}e^{\img\xi' B_1(s')}\Big]d \xi d\xi' dwdw'dudu'.
\end{align*}
After integrating in $w,w',u,u'$ on the RHS of the last integral, we have
\begin{equation*}
%	\label{e.bdG}
	|G_{\e_1,\e_2}(s,s')|\leq (2\pi)^{-2}\int_{\R^2}\E_B\Big[e^{\img\xi B_1(s)}e^{\img\xi' B_1(s')}\Big]d \xi d\xi' \leq \frac{C}{\sqrt{(s\wedge s')|s-s'|}}\in L^1([0,t]^2).
\end{equation*}
Given this, the dominated convergence theorem upgrades~\eqref{e.conG} into a convergence in $ L^1([0,t]^2) $.
This gives~\eqref{e.conz1z2} and hence completes the proof.
\end{proof}

We next turn to the distributional limit of $ X_{1,\e} $ and $ X_{2,\e} $ (defined in~\eqref{e.xeps}).
%Recall from~\eqref{e.localtime} that $ \ell(t) $ denotes the mutual intersection localtime of $ B_1 $ and $ B_2 $.
Hereafter, we use $ \Rightarrow $ to denote the weak convergence of probability laws in a designated space,
and endow the space $ C[0,\infty) $ with the topology of uniform convergence over compact subsets of $ [0,\infty) $.
\begin{proposition}\label{p.wkcon}
As $ \e\to0 $, 
\begin{equation}\label{e.con5}
	(B_1,B_2,\ell, X_{1,\e},X_{2,\e})\Longrightarrow (B_1,B_2,\ell,\sigma_* W_1,\sigma_* W_2)
	\quad
	\text{ in } (C[0,\infty))^5 .
\end{equation}
where $ W_1,W_2 $ are standard Brownian motions independent of $ (B_1,B_2) $, and $\sigma_* \in (0,\infty) $ is given in \eqref{e.sigma}.% in the following.
\end{proposition}

\begin{proof}
The proof consists of two steps.\\
\noindent\textit{Step 1:}
Instead of showing~\eqref{e.con5} directly, let us first establish
\begin{equation}\label{e.con4}
	(B_1,B_2,X_{1,\e},X_{2,\e})\Longrightarrow (B_1,B_2,\sigma_* W_1,\sigma_* W_2)
	\quad
	\text{ in } (C[0,\infty))^4.
\end{equation}
To this end, we appeal to microscopic coordinates.
That is, with $ \Xm_{\e,j} $ given in~\eqref{e.xm} (with $ B $ replaced by $ B_j $),
we have
\begin{align}
	\label{e.cnvg.scaling}
	(B_1(t),B_2(t),X_{1,\e}(t),X_{2,\e}(t))_{t\geq 0}\stackrel{\text{law}}{=} (\e B_1({\e^{-2}t}),\e B_2(\e^{-2}t), \Xm_{1,\e}(t),\Xm_{2,\e}(t))_{t\geq0}.
\end{align}
We begin by writing $ \Xm_{j,\e} $ in terms of stochastic integrals. 
Let $ D_{j,r} $ denote the Malliavin derivative with respect to $ dB_j(r) $ on $(\Sigma,\fil_B,\P_B)$,
and let $ \fil_j(r) $ denote the canonical filtration of $ B_j $.
The Clark--Ocone formula \cite[Proposition 1.3.14]{nualart2006malliavin} states that (with $\Xm_{j,\e}(t)$ having zero mean)
\begin{align*}
	\Xm_{j,\e}(t)= \e \int_0^{\e^{-2}t} \calXmco_{j,\e}(r,t) dB_j(r),
	\quad
	\calXmco_{j,\e}(r,t)(r,t) := \e^{-1}\E[ D_{j,r} \Xm_{j,\e}(t) | \fil_{j}(r) ].
\end{align*}
A direct calculation yields
\begin{align*}
	D_{j,r} \Xm_{j,\e}(t)= \e\int_0^{\e^{-2}t}\int_0^s\partial_x R(s-u,B_j(s)-B_j(u)) du ds \ind_{\{u\leq r < s\}},
\end{align*}
so
\begin{align}
	\label{e.calXmco}
	\calXmco_{j,\e}(r,t) 
	=  \int_r^{\e^{-2}t}\int_{0}^r \E_B\Big[ \partial_x R(s-u,B_j(s)-B_j(u)) \Big| \fil_{j}(r) \Big] duds.
\end{align}
%	=&\int_r^{\e^{-2}t}\int_0^r q_{s-r}\star \partial_x R(s-u,B_j(r)-B_j(u))duds,
%
%where the convolution $\star$ is in the $x$ variable.
%where $ \tilde{R}(s,x) := \int_{\R} q(s,x-y)\partial_xR(s,y) dy $.
The function $ \partial_x R(s,x) $ vanishes for all $ |s| \geq 1 $ (because $ R $ does),
so by defining
\begin{equation}\label{e.ztilde}
	\calXmcoo_j(r):=\ind\{r\geq 1\}\int_r^{r+1}\int_{r-1}^r  \E_B\Big[  \partial_x R(s-u,B_j(s)-B_j(u)) \Big| \fil_{j}(r) \Big]  duds,
\end{equation}
we have $ \calXmco_{j,\e}(r,t) = \calXmcoo_j(r) $, for all $ r\in[1,\e^{-2}t-1] $.
The latter is preferred for our purpose, because it does not depend on $ t $.
In particular, the analogous integrated process:
\begin{align*}
	t\longmapsto \Xmcoo_{j,\e}(t)
	:=
	\e \int_0^{\e^{-2}t} \calXmcoo_{j}(r) dB_j(r)
\end{align*}
is a martingale (unlike $ \Xm_{j,\e}(t) $, which is not due to the $ t $-dependence of $ \calXmco_{j,\e}(r,t) $).
Also, for each $ t\in\R_+ $, the $ L^2(\Sigma) $-distance between $ \Xmcoo_{j,\e} $ and $ \Xm_{j,\e} $ vanishes as $ \e \to 0 $:
\begin{align}
	\label{e.XmL2}
	\E_B[|\Xmcoo_{j,\e}(t)-\Xm_{j,\e}(t)|^2]
	=
	\e^2\int_0^{\e^{-2}t} \E_B[|\calXmco_{j,\e}(r,t)-\calXmcoo_{j}(r)|^2] dr\leq  C\e^2.
\end{align}
Given this, let us focus on the modified process $ \Xmcoo_{j,\e} $ instead of $ \Xm_{j,\e} $.

Now, consider the $ C([0,\infty),\R^4) $-valued process 
\[
M_\e(t) = (M_{j,\e}(t))_{j=1}^4 := (\e B_1(\e^{-2}t),\e B_2(\e^{-2}t),\Xmcoo_{1,\e}(t),\Xmcoo_{2,\e}(t)) .
\]
It is a continuous martingale, with cross-variance
\begin{align*}
	\la M_{i,\e}, M_{j,\e} \ra(t)
	=
	\left\{\begin{array}{l@{,}l}
		0	& \text{ for } (i,j)=(1,2), (3,4), (1,4), (2,3),
	\\
		\e^2 \int_0^{\e^{-2}t} \calXmcoo_{i}(s) ds & \text{ for } (i,j) = (1,3), (2,4),
	\\
		t &	\text{ for } (i,j) = (1,1), (2,2),
	\\
		\e^2 \int_0^{\e^{-2}t} \calXmcoo^2_{i}(s) ds & \text{ for } (i,j) = (3,3), (4,4).
	\end{array}\right.
\end{align*}
Further, straightforward calculations from the expression \eqref{e.ztilde} gives, with $ \partial_x R $ being an odd function in $x$, $ \E_B[\tilde{\mathcal{Z}}_{i}(s)] =0 $ and $ \E_B[\tilde{\mathcal{Z}}^2_{i}(s)] =(\sigma'_*)^2 $, for all $ s \geq 1 $, where
\begin{align}
	\label{e.sigma'}
	(\sigma'_*)^2 := \E_B[\calXmcoo^2_{i}(1)] 
	=
	\E_B \Bigg[ \E_B\Big[ \int_1^{2} \int_{0}^1 \partial_x R(s-u,B_j(s)-B_j(u)) duds\Big| \fil_{j}(1) \Big]^2\Bigg].
\end{align}
Calculating the conditional expectation in~\eqref{e.ztilde} gives
\begin{align*}
	\tilde{\calXmco}_{j}(r)
	=\ind\{r \geq 1\} \int_r^{r+1}\int_{r-1}^r \tilde{R}(s,r,u,B_j(r)-B_j(u)) duds,
\end{align*}
where $ \tilde{R}(s,r,u,x) := \int_\R q(s-r,x-y) \partial_y R(s-u,y) dy $,
and $ q(t,x) := \frac{1}{\sqrt{2\pi t}}e^{-\frac{x^2}{2t}} $ denotes the standard heat kernel.
%
%where the convolution $\star$ is in the $x$ variable.
%where $ \tilde{R}(s,x) := \int_{\R} q(s,x-y)\partial_xR(s,y) dy $.
From this expression, it is readily check that $ (\calXmcoo_{i}(s))_{s\geq 1} $ is bounded, 
stationary, and has a finite range of dependence similar to~\eqref{e.calxm.finitRg}. In particular, the process $\{B_j(r)-B_j(u): u\in[r-1,r]\}_{r\geq0}$ is ergodic.  
Consequently, Birkhoff's Ergodic theorem applied to $ (\calXmcoo_{i}(s))_{s\geq 1} $ and $ \{\calXmcoo^2_{i}(s)\}_{s\geq 1} $ gives
\begin{align*}
	\e^2 \int_0^{\e^{-2}t} \calXmcoo_{i}(s) ds \longrightarrow 0,
	\quad
	\e^2 \int_0^{\e^{-2}t} \calXmcoo^2_{i}(s) ds \longrightarrow (\sigma'_*)^2t,
\end{align*}
almost surely as $ \e\to 0 $, for any fixed $ t\in\R_+ $.
Given these properties, the martingale central limit theorem~\cite[Thm~1.4, p339]{Ethier86} yields that
$
	M \Rightarrow (B_1,B_2,\sigma'_*W_1,\sigma'_*W_2),
$
in $ (C[0,\infty))^4$.
This together with~\eqref{e.cnvg.scaling} and~\eqref{e.XmL2} gives
\begin{align}
	\label{e.fdd}
	(B_1(t),B_2(t),X_{1,\e}(t),X_{2,\e}(t)) \Longrightarrow (B_1(t),B_2(t),\sigma'_*W_1(t),\sigma'_*W_2(t)),
	\text{ in fdd.}
\end{align}

Given~\eqref{e.fdd}, it now suffices to establish the tightness of $ (B_1,B_2,X_{1,\e},X_{2,\e}) $ in $ (C[0,\infty))^4 $,
and show that $ (\sigma'_*)^2 =\sigma^2_* $.
The first step is to appeal to microscopic coordinates. 
Using~\eqref{e.xm}--\eqref{e.xm.int} we write
\begin{align*}
	X_{j,\e}(t)\stackrel{\text{law}}{=} \Xm_{j,\e}(t)
	=\e\int_1^{\e^{-2}t}  \calXm_j(s) ds + r_\e(t).
\end{align*}
Given the properties~\eqref{e.calxm.station}--\eqref{e.calxm.finitRg} of $ \calXm_j $,
a classical functional central limit theorem, see, e.g., \cite[pp 178--179]{Billingsley99}, asserts that
\begin{align}
	\label{e.cnvg.xj}
	X_{j,\e} \Longrightarrow \sigma_* W_j,
	\quad
	\text{in } C[0,\infty).
\end{align}
To apply the result in \cite{Billingsley99}, a $\varphi-$mixing condition needs to be checked \cite[Eqn (20.65)]{Billingsley99}. This is clearly satisfied in our case because $\calXm_j$ has a finite range of dependence. The convergence in \eqref{e.cnvg.xj} in particular implies the tightness of $ (B_1,B_2,X_{1,\e},X_{2,\e}) $.
Further, comparing~\eqref{e.fdd}--\eqref{e.cnvg.xj}, we see that $ \sigma^2_* = (\sigma'_*)^2 $ must holds.
(Alternatively, it is possibly to show~$ \sigma^2_* = (\sigma'_*)^2 $ by calculations from the expressions~\eqref{e.sigma} and~\eqref{e.sigma'}.)
We thus conclude~\eqref{e.con4}.

\medskip
\textit{Step 2:}
Having established~\eqref{e.con4},
our next goal is to extend the convergence result to include the localtime process $ \ell $.
First, Tanaka's formula gives
\begin{align*}
	2\ell(t)=|B_1(t)-B_2(t)|-\int_0^t \sgn(B_1(s)-B_2(s)) d(B_1-B_2)(s).
\end{align*}
\emph{Had} it been the case that the RHS were a continuous function of $ B_1-B_2 $, 
the desired result~\eqref{e.con5} would follow immediately from~\eqref{e.con4}.
We show in Lemma~\ref{l.jointcon} that, in fact, the stochastic integral $ \int_0^t \sgn(B_1(s)-B_2(s)) d(B_1-B_2)(s) $ is well-approximated 
by a sequence of continuous functions of $ B_1-B_2 $.
That is, there exists a sequence $\{f_n\}_{n\geq 1}\subset C(C[0,\infty);C[0,\infty)) $ such that
\begin{align}
	\label{e.jointcon}
	\Big\Vert \int_0^\Cdot \sgn(B_1(s)-B_2(s)) d(B_1-B_2)(s) - f_n(B_1-B_2) \Big\Vert_{C[0,\infty)} \longrightarrow 0
	\
	\text{ in probability, as }
	n \to \infty.
\end{align} 

Now, fix arbitrary bounded and continuous $ g: (C[0,\infty))^3 \to \R $ and $ h:(C[0,\infty))^2\to \R $,
and consider test functions of the type $ g\otimes h \in (C[0,\infty))^5\to \R $.
It is known that the linear span of functions of this type is dense in $ C((C[0,\infty))^5;\R) $.
Hence proving~\eqref{e.con5} amounts to proving 
\begin{align}
	\label{e.con5.}
	\E_B[g(B_1,B_2,\ell)h(X_{1,\e},X_{2,\e})]
	\longrightarrow
	\E_B[g(B_1,B_2,\ell) h(\sigma_*W_{1},\sigma_*W_{2})],
	\quad
	\text{as } \e \to 0.
\end{align}
Set $ 2\ell_n(t) := |B_1(t)-B_2(t)| - f_n(B)(t) $.
Since $ f_n $ is continuous, for each fixed $ n $, 
from~\eqref{e.con4} we have
\begin{align*}
	\E_B[g(B_1,B_2,\ell_n)h(X_{1,\e},X_{2,\e})]
	\longrightarrow
	\E_B[g(B_1,B_2,\ell_n) h(\sigma_*W_{1},\sigma_*W_{2})]
	\quad
	\text{as } \e \to 0.	
\end{align*}
On the other hand, with $ g $ being bounded and continuous, by~\eqref{e.jointcon}, we have
\[
 \E_B[|g(B_1,B_2,\ell_n)-g(B_1,B_2,\ell)|] \to 0,
 \] as $ n\to\infty $.
From these the desired result~\eqref{e.con5.} follows. The proof is complete.
\end{proof}

\begin{lemma}\label{l.jointcon}
The claim~\eqref{e.jointcon} holds for a sequence $ \{f_n\}_{n\geq 1} \subset C(C[0,\infty);C[0,\infty)) $.
\end{lemma}
\begin{proof}
Set $ \ov{B}(t) := B_1(t)-B_2(t) $ and $ U(t):= \int_0^t \sgn(\ov{B}(s)) d\ov{B}(s) $ to simplify the notation.
We begin by constructing the continuous function $ f_n $.
Set $ \zeta(x) := x \ind\{|x|\leq 1\}+ \sgn(x) \ind\{ |x|>1\}$, $\zeta_n(x) := \zeta(n^{1/4}x)$,
and define
\begin{align*}
	f_n(y)(t) 
	:=
	\int_0^t \sum_{k=0}^\infty \zeta_n(y(\tfrac{k}{n})) \ind_{[\frac{k}{n},\frac{k+1}{n})}(s) dy(s)
	:=
	\sum_{k=0}^\infty \zeta_n(y(\tfrac{k}{n})) \big( y(\tfrac{k+1}{n}\wedge t)-y(\tfrac{k}{n} \wedge t) \big).
\end{align*}
Indeed, $ f_n $ is continuous for each fixed $ n $.
Fix an arbitrary $T>0$. 
Using Doob's $ L^2 $-martingale inequality and It\^{o} isometry, we calculate
\begin{equation}
\begin{aligned}
    \E_B \Big[ \sup_{[0,T]} |U(t)-f_n(\ov{B})(t)|^2\Big]
	&\leq
	C \int_0^T \E_B\Big[\Big| \sum_{k=0}^\infty \zeta_n(\ov{B}(\tfrac{k}{n})) \ind_{[\frac{k}{n},\frac{k+1}{n})}(t)-\sgn(\ov{B}(t)) \Big|^2\Big] dt
\\
	&=
	C \sum_{k=0}^\infty \int_{I_{k,n}(T)} \E_B[|  \zeta_n(\ov{B}(\tfrac{k}{n})) -\sgn(\ov{B}(t)) |^2] dt,
\end{aligned}
\end{equation}
where $ I_{k,n}(T) := [\frac{k}{n},\frac{k+1}{n}) \cap [0,T] $.
Set 
$
	V_{k,n} := \sup_{s\in[0,\frac{1}{n}]} |\ov{B}(s+\frac{k}{n})-\ov{B}(\frac{k}{n})|.
$
On the interval $t\in I_{k,n}(T) $, we have $\zeta_n(\ov{B}(\frac{k}{n}))=\sgn(\ov{B}(t)) $ whenever
$ |\ov{B}(\frac{k}{n})| > n^{-\frac14} $ and $ V_{k,n} < n^{-\frac14} $.
Hence
\begin{align*}
    \E_B \Big[ \sup_{[0,T]} |U(t)-f_n(\ov{B})(t)|^2\Big]
	\leq
	C \sum_{k=0}^{[nT]} \Big( \P_B\big[ V_{k,n} \geq n^{-\frac14}\big] +\P_B\big[|\ov{B}(\tfrac{k}{n})|\leq n^{-\frac14}\big] \Big) \frac{1}{n}.
\end{align*}
By the scaling property of Brownian motion and the reflection principle,
we have that $ V_{k,n}\stackrel{\text{law}}{=} \sqrt{2/n}|Z| $ and that $ \ov{B}(\tfrac{k}{n}) \stackrel{\text{law}}{=} \sqrt{2k/n}Z $,
where $ Z $ is a standard Gaussian.
This gives
\begin{align*}
    \E_B \Big[ \sup_{[0,T]} |U(t)-f_n(\ov{B})(t)|^2\Big]
	\leq
	C \sum_{k=0}^{[nT]} \Big( \P_B\big[ |Z| \geq 2^{-\frac12} n^{\frac14} \big] +\P_B\big[|Z|\leq \tfrac{n^{1/4}}{(2k)^{1/2}} \big] \Big) \frac{1}{n},
\end{align*}
with $ \tfrac{n^{1/4}}{(2k)^{1/2}}:=\infty$ when $k=0$. It is now readily verified that the last expression tends to $ 0 $ as $ n\to\infty $.
From this the desired result follows: $ \sup_{[0,T]} |U(t)-f_n(\ov{B})(t)| \to 0 $ in probability, as $ n\to\infty $, for each fixed $ T $.
\end{proof}

\section{Proof of Theorem~\ref{t.mainth}}\label{s.proof}

Let us first establish the boundedness of moments of $ u_\e(t,x) $.
Set $ \lambda_n := \frac{n(n+1)}{2} $.
With the initial condition $ u_0 $ being bounded, applying H\"{o}lder's inequality 
(with exponents $ (\lambda_n,\ldots,\lambda_n) $) in the formula~\eqref{e.nmom}, we have
\begin{align*}
	\E [|u_\e(t,x)|^n] 
	&
	\leq
	C
	\Bigg(
		\prod_{j=1}^n \E_B\big[ \exp\big( \lambda_n\big(X_{j,\e}(t) - \tfrac{1}{2}\sigma_*^2 t + r_{\e}(t) \big)\big) \big]
		\prod_{1\leq i<j\leq n} \E_B\big[ \exp\big( \lambda_nY_{i,j,\e}(t) \big) \big]
	\Bigg)^{1/\lambda_n}
\\
	&
	\leq
	C
	\Bigg( \E_B\big[ \exp\big( \lambda_nX_{\e}(t) \big) \big] \Bigg)^{n/\lambda_n}
	\Bigg( \E_B\big[ \exp\big( \lambda_nY_{i,j,\e}(t) \big) \big] \Bigg)^{n(n-1)/2\lambda_n}.
\end{align*}
Using the exponential moment bounds from Proposition~\ref{p.expin}, we obtain
$ \sup_{\e\in(0,1)}\E [|u_\e(t,x)|^n] <\infty $.
That is, moments of $ u_\e(t,x) $ are bounded uniformly in $ \e $.
This reduces proving~\eqref{e.main} for all $ n\geq 1 $ to proving~\eqref{e.main} for \emph{just one} $ n\geq 1 $, since it implies the convergence in probability, and combining with the uniform integrability of $|u_\e(t,x)|^n$, we will have the convergence in $L^n(\Omega)$. We henceforward consider $ n=2 $.

%
%We first identify the two renormalizing constant. The following lemma defines $c_1$.
%\begin{lemma}\label{l.c1}
%There exists $c_1>0$ such that 
%\[
%\int_0^t \int_0^s\E_B[R_\e(s-u,B_s-B_u)]duds=\frac{c_1t}{\e}+O(\e)
%\]
%as $\e\to0$.
%\end{lemma}
%\begin{proof}
%%By the scaling property of the Brownian motion and a change variables, we have 
%Similar to the proof of Proposition~\ref{p.expin}, we have
%\[
%\begin{aligned}
%\int_0^t\int_0^s\E_B[R_\e(s-u,B_s-B_u)]duds=&\e\int_0^{t/\e^2} \int_0^s\E_B[R(s-u,B_s-B_u)]duds\\
%=&\e\int_0^{t/\e^2} ds \left(\int_0^s \E_B[R(u,B_s-B_{s-u})]du\right).
%\end{aligned}
%\]
%Since $R(u,\Cdot)=0$ for $|u|\geq 1$, %is compactly supported, there exists $M>0$ such that
%\[
%\e\int_0^{t/\e^2} ds \left(\int_0^s \E_B[R(u,B_s-B_{s-u})]du\right)=\e\int_1^{t/\e^2} ds \left(\int_0^1 \E_B[R(u,B_u)]du\right)+O(\e).
%\]
%Define 
%\begin{equation}\label{e.c1}
%c_1=\int_0^1 \E_B[R(u,B_u)]du,
%\end{equation}
%we have 
%\[
%\int_0^t\int_0^s\E_B[R_\e(s-u,B_s-B_u)]duds=c_1\e(\frac{t}{\e^2}-1)+O(\e)=\frac{c_1t}{\e}+O(\e).
%\]
%The proof is complete.
%\end{proof}
%
%Let 
%\begin{equation}\label{e.c2}
%c_2=\frac{\sigma^2}{2},
%\end{equation}
%with $\sigma^2$ defined in \eqref{e.defsigma}.

Let us first identify the limit of $ \E[u_{\e_1}(t,x)u_{\e_2}(t,x)] $.
Recall from~\eqref{e.localtime} that $ \ell(t) $ denotes the mutual intersection localtime of $B_1,B_2$,
and that $ \U $ denotes the solution of the \ac{SHE}~\eqref{e.she}.
\begin{proposition}\label{p.cauchy}
We have
\begin{align*}
	\lim_{\e_1,\e_2\to 0} \E[u_{\e_1}(t,x)u_{\e_2}(t,x)]
	=
	\E_B\Big[u_0(x+B_1(t))u_0(x+B_2(t)) \exp(\ell(t))\Big].
%	=
%	\E[\, \U(t,x)^2].
\end{align*}
\end{proposition}
\begin{proof}
%We show only the first equality, as the second is a classical result for stochastic heat equation.
The starting point of the proof is the formula~\eqref{e.2mom}:
\begin{align}
	\E[u_{\e_1}(t,x)u_{\e_2}(t,x)]
	=
	\E_B\Big[ \prod_{j=1}^2 u_0(x+ B_j(t)) 
	\exp\Big( Y_{\e_1,\e_2}(t)+\sum_{j=1}^2\Big(X_{j,\e_j}(t) - \frac{1}{2}\sigma_*^2 t + r_{\e_j}(t) 
	\Big)\Big) \Big].
\end{align}
By virtue of Propositions~\ref{p.ycnvg}--\ref{p.wkcon}, we have
\[
(B_1(t),B_2(t),Y_{\e_1,\e_2}(t),X_{1,\e_1}(t),X_{2,\e_2}(t))\Rightarrow (B_1(t),B_2(t),\ell(t),\sigma_*W_1(t),\sigma_*W_2(t))
\]
in distribution.
%we can replace $ Y_{\e_1,\e_2}(t) $ by $ \ell(t) $ in~\eqref{e.2mom} up to a vanishing error. 
%That is,
%\begin{align*}
%	\Bigg|
%		\E[u_{\e_1}(t,x)u_{\e_2}(t,x)]
%		-
%		\E_B\Big[ \prod_{j=1}^n u_0(x+\e B_j(\e^{-2}t)) 
%		\exp\Big( \ell(t)+\sum_{j=1}^2\Big(X_{j,\e}(t) - \tfrac{1}{2}\sigma_*^2 t + r_{j,\e}(t) 
%		\Big) \Big]
%	\Bigg|
%	\longrightarrow 0.
%\end{align*}
The proof is complete by invoking Propositions~\ref{p.expin}.% and \ref{p.wkcon}, we have
%\begin{align*}
%	\Bigg|
%		\E_B\Big[ \prod_{j=1}^n u_0&(x+B_j(t)) 
%		\exp\Big( \ell(t)+\sum_{j=1}^2\Big(X_{j,\e}(t) - \tfrac{1}{2}\sigma_*^2 t + r_{j,\e}(t) 
%\\
%		&-
%		\E_B\Big[ \prod_{j=1}^n u_0(x+B_j(t)) 
%		\exp\Big( \ell(t)+\sum_{j=1}^2\Big(W_{j,\e}(t) - \tfrac{1}{2}\sigma_*^2 t 
%		\Big) \Big]
%	\Bigg|
%	\longrightarrow 0.
%\end{align*}
%the desired result follows.
\end{proof}

Now, with
\begin{align*}
	\E[(u_{\e_1}(t,x)-u_{\e_2}(t,x))^2]= \E[u_{\e_1}(t,x)u_{\e_1}(t,x)] - 2\E[u_{\e_1}(t,x)u_{\e_2}(t,x)] + \E[u_{\e_2}(t,x)u_{\e_2}(t,x)],
\end{align*}
Proposition~\ref{p.cauchy} has an immediate corollary:
\begin{corollary}\label{c.cauchy}
The sequence $ \{u_\e(t,x)\}_{\e\in(0,1)} $ is Cauchy in $ L^2(\Omega) $.
\end{corollary}

Given this result, it suffices to identify the unique limit of $ u_\e(t,x) $ in $L^2(\Omega)$.
We achieve this by \emph{Wiener chaos expansion}.
Fix $ (t,x)\in\R_+\times\R $ hereafter, and denote $ R^k_< :=\{(s_1,\ldots,s_k)\in\R^k: s_1<\ldots<s_k\} $. % denote the Weyl chamber
Given any $ f\in L^2(R^k_<\times \R^k) $, we consider the $ k $-th order multiple stochastic integral
\begin{align*}
	I_k(f)
	:=
	\int_{\R^k_<\times\R^k} f(s_1,\ldots,s_k,y_1,\ldots,y_k) \prod_{i=1}^n \xi(s_i,y_i) ds_idy_i.
\end{align*}
Let $ q(s,y) = \frac{1}{\sqrt{2\pi s}} e^{-y^2/2s} $ denotes the standard heat kernel, the solution $ \U $ of the \ac{SHE} permits the chaos expansion (we omit the dependence on $(t,x)$)
\begin{align}
	\label{e.chaos.U}
	\U(t,x) = \sum_{k=0}^\infty I_k(f_k),
	\quad
	f_k := \ind_{\{0<s_1<\ldots<s_k<t\}}\int_\R u_0(y_0)\prod_{i=0}^k q(s_{i+1}-s_i,y_{i+1}-y_{i}) dy_0,
\end{align}
under the convention $ y_{k+1}:=x $, $ s_0:=0 $, and $ s_{k+1}:= t $. %, and we denote $(0,t)_<^k:=\{0<s_1<\ldots<s_k<t\}$. 
For $ u_\e $, a similar expansion also exists: using the Stroock formula~\cite[Eqn (7), p3]{Stroock87}, we arrive at
\begin{align}
	\label{e.chaos.u}
	u_\e(t,x)=\sum_{k=0}^\infty I_k(f_{k,\e}),
	\quad
	f_{k,\e}:=\E[ D^k u_\e(t,x)].
\end{align}
\rev{(Alternatively, this formula~\eqref{e.chaos.u} can also be obtained from the Wick exponential.)}
Here $ D $ denotes the Malliavin derivative with respect to $\xi$ on $(\Omega,\fil,\P)$.
To calculate the chaos coefficient $ f_{k,\e} $, we set
\begin{align}\label{e.defPsi}
	\Psi_{\e,B}(r,y):=\int_0^t \phi_\e(t-s-r,x+B(s)-y)ds,
\end{align}
and rewrite the Feynman--Kac formula \eqref{e.fk} as 
\begin{align*}
	u_\e(t,x)=\E_B\Big[ u_0(x+B(t)) \exp\Big( \int_{\R^{2}} \Psi_{\e,B}(r,y)\xi(r,y) dydr -c_\e t\Big)\Big].
\end{align*}
From this expression we calculate
\begin{equation}
\begin{aligned}
	\label{e.chaos.fe}
	&f_{k,\e}(r_1,\ldots,r_k,y_1,\ldots,y_k)\\
	&=\E_B\Big[u_0(x+B(t))\exp\Big(\int_0^t\int_0^{s} R_\e(s-u,B(s)-B(u))duds-c_\e t\Big) \prod_{i=1}^k\Psi_{\e,B}(r_i,y_i)\Big].
\end{aligned}
\end{equation}

Denoting the $L^2(\Omega)-$limit of $u_\e (t,x)$ by $\mathscr{U}(t,x)$, and the chaos expansion of $\mathscr{U}(t,x)$ is written as 
\begin{equation}\label{e.chaoslimit}
\mathscr{U}(t,x)=\sum_{k=0}^\infty I_k(\tilde{f}_k).
\end{equation}

The following lemma completes the proof of Theorem~\ref{t.mainth}.
\begin{lemma}\label{l.chaos}
$\mathscr{U}(t,x)=\mathcal{U}(t,x)$ in $L^2(\Omega)$.%For each fixed $ k\geq 0 $, $ f_{k,\e} \to f_{k} $ in $ L^2([0,t]^k_<\times\R^k) $.
\end{lemma}
\begin{remark}
\rev{The major component of the following proof is to establish the convergence $ f_{k,\e}\to f_k $ in $ L^2(\R^k_<\times\R^k) $.
To set up the premise of the proof, we first give a heuristic explanation why the convergence should hold.
Under current notations, Proposition~\ref{p.wkcon} gives the following weak convergence in $ C[0,\infty)^2 $,
\begin{align}
	\label{e.wkcon}
	\Big(\int_0^t\int_0^{s} R_\e(s-u,B(s)-B(u))duds-c_\e t, B(t)\Big) \Longrightarrow  \Big(\sigma_* W(t)-\tfrac12\sigma_*^2t, B(t)\Big),
\end{align}
where $ W $ and $ B $ are \emph{independent} standard Brownian motions.
This implies
\begin{align}
	\label{e.wkcon:}
	\E_B\Big[u_0(x+B(t))\exp\Big(\int_0^t\int_0^{s} R_\e(s-u,B(s)-B(u))duds-c_\e t\Big) F(B)\Big]
	\longrightarrow
	\E_B\big[u_0(x+B(t))F(B)\big],
\end{align}
for any continuous, bounded test function $ F: C[0,\infty) \to \R $.
Referring to~\eqref{e.chaos.fe} and \eqref{e.chaos.U},
together with $\phi_\e(t,x)\to \delta(t,x)$,
we can informally view the convergence $ f_{k,\e}\to f_k $ as a generalization of~\eqref{e.wkcon:}
where the test function depend on $ \e $.}

\rev{To \emph{prove} $ f_{k,\e}\to f_k $, the weak convergence~\eqref{e.wkcon} does not suffice.
This is so especially because the test function $ \phi_\e(t,x) $ (which approximates the Dirac function)
probes \emph{small}-scales that are not compatible with the topology of the weak convergence~\eqref{e.wkcon}.
One possible proof is to establish a local version~\eqref{e.wkcon} that is commensurate with the scale of $ \phi_\e(t,x) $.
Doing so requires much technical effort.
Instead, we \emph{circumvent} this technical issue by testing $ f_{k,\e} $ against a smooth test function $ g $,
after which the weak convergence~\eqref{e.wkcon} applies.}
%It is natural to ask whether one can directly prove the convergence of . %, rather than first proving the convergence of $u_\e(t,x)$ in $L^2(\Omega)$ which implies the convergence of $f_{k,\e}$, then identifying the limit by considering $\lim_{\e\to0}\la f_{k,\e},g\ra$ with test function $g$. 
%Take $k=1$ as an example, we have the coefficient
%\begin{equation}\label{e.10161}
%f_{1,\e}(r,y)=\E_B\Big[u_0(x+B(t))\exp\Big(\int_0^t\int_0^{s} R_\e(s-u,B(s)-B(u))duds-c_\e t\Big) \int_0^t \phi_\e(t-s-r,x+B(s)-y)ds\Big],
%\end{equation}
%and the convergence in distribution of
%\begin{equation}\label{e.1016}
%\Big(\int_0^t\int_0^{s} R_\e(s-u,B(s)-B(u))duds-c_\e t, B(s)\Big)\Rightarrow  \Big(\sigma_* W_t-\tfrac12\sigma_*^2 t, B(s)\Big).
%\end{equation}
%If the exponential factor in \eqref{e.10161} is viewed as a weight on the Wiener measure and we consider $B$ as sampled from the weighted measure, the convergence in \eqref{e.1016} shows that the asymptotic behavior of $B$ is still a standard Brownian motion. Nevertheless, since $\phi_\e(t,x)\to \delta(t,x)$, to pass to the limit of $f_{1,\e}$, we need a local central limit theorem which is stronger than \eqref{e.1016}. For this reason, we find our approach more convenient.
\end{remark}
\begin{proof}
Given the chaos expansions in \eqref{e.chaos.U} and \eqref{e.chaoslimit}, it suffices to show $\tilde{f}_k=f_k$. Since $u_\e(t,x)\to \mathscr{U}(t,x)$ in $L^2(\Omega)$, using the orthogonality of the chaos, i.e.,
\begin{align*}
	\E\big[(u_{\e}-\mathscr{U})^2\big]
	=
	\sum_{k=0}^\infty \int_{\R^k_<\times\R^k} (f_{k,\e}-\tilde{f}_k)^2 dsdy,
\end{align*}
we see that, for each $ k \geq 0 $, $f_{k,\e} \to \tilde{f}_k$ in $ L^2(\R^k_<\times\R^k) $. 
Fix arbitrary $ g\in C_c^\infty(\R^{2n}) $, we consider 
\begin{align}
	\la f_{k,\e},g\ra
	:=\int_{\R^k_<\times\R^k} f_{k,\e}(r_1,\ldots,r_k,y_1,\ldots,y_k) g(r_1,\ldots,r_k,y_1,\ldots,y_k)drdy.
\end{align}
To prove $\tilde{f}_k=f_k$, it suffices to show $\la f_{k,\e},g\ra\to \la f_{k},g\ra$ as $\e\to0$. 

The formula~\eqref{e.chaos.fe} yields
\begin{equation}\label{e.fg}
\begin{aligned}
\la f_{k,\e},g\ra=\int_{\R_<^k\times \R^k} \E_B\Big[ &u_0(x+B(t)) e^{\int_0^t\int_0^s R_\e(s-u,B(s)-B(u))duds-c_\e t}\\
&\Big(\int_{[0,t]^k}\prod_{i=1}^k \phi_\e(t-s_i-r_i,x+B(s_i)-y_i)ds\Big) \Big] g(r_1,\ldots,r_k,y_1,\ldots,y_k)drdy.
\end{aligned}
\end{equation}
With $ \phi_\e(\Cdot,\Cdot) := \e^{-3}\phi(\e^{-2}\Cdot,\e^{-1}\Cdot) $,
we perform a change of variables $ r_i\mapsto \e^2 r_i'+t-s_i $, $ y_i\mapsto \e y_i'+x+B(s_i)
$ to rewrite the last expression as
\[
\begin{aligned}
\la f_{k,\e},g\ra=\int_{\R^{3k}}&\ind_{A_\e\cap A'}(r',s)\prod_{i=1}^k \phi(-r_i',-y_i')\E_B\Big[ u_0(x+B(t)) e^{\int_0^t\int_0^s R_\e(s-u,B(s)-B(u))duds-c_\e t} \\
&\times g\big(\e^2r_1'+t-s_1,\ldots,\e^2r_k'+t-s_k,\e y_1'+x+B(s_1),\ldots,\e y_k'+x+B(s_k)\big)\Big]dr'dy'ds,
\end{aligned}
\]
where $A_\e:=\{\e^2 r_1'+t-s_1<\ldots<\e^2 r_k'+t-s_k\}$ and $ A' := \{(s_1,\ldots,s_k)\in[0,t]^k\} $ 
translate the constraints on the old variables into the new ones.
In order to pass to the limit, we note that, by Proposition~\ref{p.wkcon}, \eqref{e.xeps}, \eqref{e.c1} and our choice of $c_\e$, for any fixed $(s_1,\ldots,s_k)\in[0,t]^k$,
\[
\Big(\int_0^t\int_0^s R_\e(s-u,B(s)-B(u))duds-c_\e t, B(s_1),\ldots,B(s_k),B(t)\Big)\Rightarrow \Big(\sigma_* W(t)-\tfrac12\sigma_*^2 t, B(s_1),\ldots,B(s_k),B(t)\Big).
\]
In addition, for any fixed $(r_1',\ldots,r_k')\in\R^k$, 
$
\ind_{A_\e\cap A'}(r',s)\to \ind_{\{0<s_k<\ldots<s_1<t\}}.
$
By applying Proposition~\ref{p.expin} and the dominated convergence theorem, we arrive at 
\begin{align*}
	\la f_{k,\e},g\ra\longrightarrow \int_{\R^{3k}}&\ind_{\{0<s_k<\ldots<s_1<t\}}\prod_{i=1}^k \phi(-r_i',-y_i')
\\
	&\times \E_B\Big[u_0(x+B(t)) e^{\sigma_*W(t)-\frac12\sigma_*^2 t}g(t-s_1,\ldots,t-s_k,x+B(s_1),\ldots,x+B(s_k))\Big] dr'dy'ds.
\end{align*}
In the last expression, integrate over $ (r_i,y_i) $ using $\int \phi drdy=1$,
and perform a change of variables $ t-s_i\mapsto s_i $.
We see that it equals
\begin{align*}
	\int_{\R_<^k\cap[0,t]^k} \E_B\Big[u_0(x+B(t)) g(s_1,\ldots,s_k,x+B(t-s_1),\ldots,x+B(t-s_k))\Big]ds=\la f_k, g\ra.
\end{align*}
The proof is complete.\end{proof}

\begin{remark}
In the discrete setting, the convergence to SHE from the partition function of a random polymer with a weak randomness was proved in Alberts-Khanin-Quastel \cite{alberts2014intermediate}, also based on a chaos expansion.
\end{remark}
%\begin{align}
%	\label{e.innerp}
%	\la f_{k,\e},g\ra
%	=\int_{[0,t]^k_<\times\R^k}
%	\E_B\Big[ u_0(x+B(t)) &\exp\Big( \int_0^t\int_0^s R_\e(s-u,B(s)-B(u))duds-c_\e t \Big) 
%\\	
%	&\prod_{i=1}^k \phi(s_i,y_i)Z_{\e,B}(s,y) \Big] dsdy,
%\end{align}
%where 
%\begin{align*}
%	Z_{\e,B}(s,y)=\int_{[0,t]^k_<}g(\e^2 s_1-r_1+t,\e y_1+B(r_1)+x, \ldots, \e^2s_n-r_n+t,\e y_n+B(r_n)+x)dr.
%\end{align*}
%With $ g \in C^\infty_c $, replacing $ Z_{\e,B}(s,y) $ with $ Z_{0,B}(s,y) $ introduces an error that vanishes as $ \e \to 0 $.
%Furthermore, the resulting expression (after replacement) is a continuous function of $ (B,X) $.
%Hence, applying Propositions~\ref{p.expin} and \ref{p.wkcon} gives 
%\begin{align*}
%	\la f_{k,\e},g\ra
%	\longrightarrow&
%	\int\limits_{[0,t]^k_<\times\R^k}\E_B\Big[ u_0(x+B(t)) \exp\Big( \int_0^t \sigma_*W(u)-\tfrac12 \sigma_*^2 t \Big) 
%	\prod_{i=1}^k \phi(s_i,y_i)Z_{0,B}(s,y) \Big] dsdy
%\\
%	=&
%	\int\limits_{[0,t]^k_<\times\R^k}\E_B\Big[ u_0(x+B(t))  
%	\prod_{i=1}^k \phi(s_i,y_i)Z_{0,B}(s,y) \Big] dsdy	
%	=
%	\la f_{k},g\ra.
%\end{align*}
%As this holds for all $ g\in C^\infty_c(\R^{2k}) $, the desired result follows.

\appendix
\section{Homogenization and stochasticity}
\label{s.transition}

In this appendix we investigate the behaviors of~\eqref{e.maineq} at different scales.
Let $ u_0(x)\in C_b(\R) $ be fixed as in the rest of the article.
For $ \alpha >0 $, set
\begin{align*}%\label{e.defxi}
	\xi_{\e,\alpha}(t,x)=\int_{\R^2} \phi_{\e,\alpha}(t-s,x-y)\xi(s,y)dyds,
	\quad
	\phi_{\e,\alpha}(t,x)=\e^{-\frac32 \alpha}\phi(\e^{-\alpha}t,\e^{-\frac{\alpha}{2}}x),
\end{align*}
and consider the solution $ v_{\e,\alpha} $ of
\begin{align}
	\label{e.vea}
	\partial_t v_{\e,\alpha} = \tfrac12 \partial_{xx} v_{\e,\alpha} +\e^{\frac12-\frac{\alpha}{4}} \xi_{\e,\alpha} v_{\e,\alpha},
	\quad
	v_{\e,\alpha}(0,x) = u_0(x).
\end{align}
Referring to~\eqref{e.defxi}, we see that $ \xi_{\e,2} = \xi_{\e} $.
For $ \alpha=2 $, \eqref{e.vea} is the same as \eqref{e.maineq} up to a centering by $ -c_\e u_{\e} $.
We thus view~\eqref{e.vea} as a generalization of~\eqref{e.maineq} to scales $ \alpha>0 $.
To see why the specific choice of prefactor~$ \e^{\frac12-\frac14\alpha} $ in~\eqref{e.vea} is relevant,
perform a change of variable $ v(t,x) := v_{\e,\alpha}(\e^{\alpha}t, \e^{\alpha/2} x) $ in~\eqref{e.vea}
to bring the equation into the `microscopic' coordinates.
Using the scaling properties $\sqrt{AB}\, \xi(A\Cdot,B\Cdot) \stackrel{\text{law}}{=}  \xi(\Cdot,\Cdot) $ of $ \xi $,
we see that $ v $ solves
\begin{align}
	\label{e.micro}
	\partial_t v=\tfrac12 \partial_{xx} v+\sqrt{\e} \eta(t,x) v,
\end{align}
where
$
%\begin{align*}
	\eta(t,x) = \int_{\R^2} \phi(t-s,x-y)\tilde{\xi}(s,y)dyds,
%\end{align*}
$
for some $ \tilde{\xi} \stackrel{\text{law}}{=} \xi $.
That is, the equation~\eqref{e.vea} encodes the behavior of~\eqref{e.micro} (which is $ \alpha $-independent)
at the scale $(t,x)\sim (\e^{-\alpha},\e^{-\alpha/2}) $.

Theorem~\ref{t.mainth} yields that $ v_{\e,2}\exp(-c_\e t) \to \U $.
On the other hand, a homogenization result was proved in \cite{bal2011convergence,Pardoux12} at the scale $ \alpha=1 $,
$
	v_{\e,1}(t,x) \to \bar{v}(t,x) \exp(c_*t),
$
where $ \bar{v} $ solves the unperturbed heat equation 
\[
\partial_t\bar{v} =\tfrac12 \partial_{xx} \bar{v}, \quad \bar{v}(0,x)=u_0(x).
\]
In the following, we establish an analogous homogenization result for $ \alpha \in [1,2) $,
together with a Gaussian fluctuation result.
\begin{proposition}\label{p.between}
Fix $ \alpha\in [1,2) $. For any given $(t,x)\in \R_+\times \R$, we have 
\begin{align}
	\label{e.homo1}
	&v_{\e,\alpha}(t,x) \exp(-\tfrac{c_*t}{\e^{\alpha-1}}) \longrightarrow \bar{v}(t,x) \quad \text{in probability},
\\
	&\E \Big[ \Big| \e^{-\frac{2-\alpha}{4}} \big( v_{\e,\alpha}(t,x)-\E[v_{\e,\alpha}(t,x)] \big) \exp(-\tfrac{c_*t}{\e^{\alpha-1}}) - \mathscr{V}(t,x) \Big|^2 \Big] 
	\longrightarrow 0.
\end{align}
where %$ \bar{v} $ solves the heat equation $ \partial_t\bar{v} =\frac12 \partial_{xx} \bar{v} $
%with initial condition $ \bar{v}(0,x) = u_0(x) $,
%and 
$\mathscr{V}$ solves the Edwards-Wilkinson equation
\begin{align}
	\label{e.ew}
	\partial_t \mathscr{V}=\tfrac12\partial_{xx}\mathscr{V}+\bar{v}\xi, \quad \mathscr{V}(0,x) = 0.
\end{align}
\end{proposition}

The result shows that if we start the microscopic dynamics \eqref{e.micro} with the randomness of size $\sqrt{\e}$, then the small Gaussian fluctuations prevail in $t\sim \e^{-\alpha}$ for any $\alpha<2$. As we increase the time scale to $\alpha=2$, the random fluctuations become of order $O(1)$ and is described by the SHE.

\begin{proof}[Sketch of proof]
Compared with Theorem~\ref{t.mainth} (i.e., $\alpha=2$), 
the scale $ \alpha \in [1,2) $ considered here is easier to analyze, so we only sketch the proof. 
Fix $ \alpha\in[1,2) $ and set $ \beta := \tfrac12-\tfrac{\alpha}{4}>0 $.
Apply the Stroock formula in a similar way as we established \eqref{e.chaos.u} and~\eqref{e.chaos.fe}, here we have 
\begin{align}
	\label{e.chaos.u.}
	v_\e(t,x) \exp(-\tfrac{c_*t}{\e^{\alpha-1}}) = \sum_{k=0}^\infty \e^{\beta k} I_k(\bar{f}_{k,\e^{\alpha/2}}),
\end{align}
where
\begin{align}
\begin{aligned}
	\label{e.chaos.fe.}
	\bar{f}_{k,\delta}(r_1,\ldots,r_k,y_1,\ldots,y_k)=\E_B\Big[u_0(x+B(t))\exp\Big(\e^{2\beta}X_{\delta}(t)+r_\e(t)\Big) \prod_{i=1}^k\Psi_{\delta,B}(r_i,y_i)\Big].
\end{aligned}
\end{align}
With $ \e^{\beta k} \leq \e^{\beta} \to 0 $, for $ k \geq 1 $, it is not hard to show that
\begin{align*}
	\sum_{k=1}^\infty \e^{\beta k} I_k(\bar{f}_{k,\e^{\alpha/2}}) \longrightarrow 0
	\quad
	\text{in } L^2(\Omega),
\end{align*}
and that $ \bar{f}_{0,\e^{\alpha/2}}\to \E_B[u_0(x+B(t))]=\bar{v}(t,x) $. This concludes the first claim. 

Moving onto the second claim regarding random fluctuations, we write
\begin{equation}\label{e.10162}
\e^{-\beta}(v_\e(t,x)-\E[v_\e(t,x)]) \exp(-\tfrac{c_*t}{\e^{\alpha-1}})=I_1(\bar{f}_{1,\e^{\alpha/2}})+\sum_{k=2}^\infty \e^{\beta (k-1)} I_k(\bar{f}_{k,\e^{\alpha/2}}).
\end{equation}
Again, With $ \e^{\beta (k-1)} \leq \e^{\beta} \to 0 $, for $ k\geq 2 $,
it is not hard to show that the second term on the RHS goes to zero in $L^2(\Omega)$. 
This being the case, we focus on the first order chaos $ I_1(\bar{f}_{1,\e^{\alpha/2}}) $. 
Using a similar argument as in the proof of Lemma~\ref{l.chaos}, here we have
\begin{align*}
	\bar{f}_{1,\e^{\alpha/2}}(r,y)\to q(t-r,x-y)\int_\R u_0(z)q(r,y-z)dz=q(t-r,x-y)\bar{v}(r,y), \quad \text{in } L^2(\R^2),
\end{align*}
which then yields
\begin{align*}
	I_1(\bar{f}_{1,\e^{\alpha/2}}) \longrightarrow \int_0^t\int_\R q(t-r,x-y)\bar{v}(r,y)\xi(r,y)drdy,
	\quad
	\text{in } L^2(\Omega).
\end{align*}
The expression $ \int_0^t\int_\R q(t-r,x-y)\bar{v}(r,y)\xi(r,y)drdy $ is exactly $ \mathscr{V}(t,x) $.
We hence conclude the second claim.
\end{proof}

\end{document}